\documentclass[letterpaper, 12pt]{amsart}

\usepackage[T1]{fontenc}

\usepackage[margin=1in]{geometry}

\usepackage{amscd}

\usepackage{amsmath}

\usepackage{amsfonts}

\usepackage{amssymb}

\usepackage{amsthm}

\usepackage{arydshln}

\usepackage{fancyhdr}

\usepackage{verbatim}

\usepackage{tikz-cd}

\usepackage[all]{xy}

\usepackage{color}

\usepackage[colorlinks=true, linkcolor=blue, citecolor=blue,
pagebackref=true]{hyperref}

\usepackage{mathpazo}
\usepackage[euler-digits, euler-hat-accent]{eulervm}

\usepackage{mathtools} \usepackage{etoolbox}
\patchcmd{\section}{\scshape}{\bfseries}{}{} \makeatletter
\renewcommand{\@secnumfont}{\bfseries} \makeatother
\theoremstyle{definition}

\theoremstyle{plain} \newtheorem{theorem}{Theorem}[section]

\theoremstyle{plain} \newtheorem{lemma}[theorem]{Lemma}

\theoremstyle{plain} \newtheorem{proposition}[theorem]{Proposition}

\theoremstyle{plain} \newtheorem{conjecture}[theorem]{Conjecture}

\theoremstyle{plain} \newtheorem{corollary}[theorem]{Corollary}

\theoremstyle{remark} \newtheorem*{remark}{Remark}

\theoremstyle{definition} 

\theoremstyle{definition} \newtheorem*{definition*}{Definition}

\theoremstyle{definition} \newtheorem{question}[theorem]{Question}

\theoremstyle{remark} 

\makeatletter \renewenvironment{proof}[1][\proofname]
{\par\pushQED{\qed}\normalfont\topsep6\p@\@plus6\p@\relax\trivlist\item[\hskip\labelsep\bfseries#1\@addpunct{.}]\ignorespaces}{\popQED\endtrivlist\@endpefalse}
\makeatother

\newcommand{\EE}{\mathbb{E}}

\newcommand{\PP}{\mathbb{P}}

\newcommand{\RR}{\mathbb{R}}

\newcommand{\QQ}{\mathbb{Q}}

\newcommand{\NN}{\mathbb{N}}

\newcommand{\ZZ}{\mathbb{Z}}

\newcommand{\calA}{\mathcal{A}}

\newcommand{\calF}{\mathcal{F}}

\newcommand{\calP}{P}

\newcommand{\calW}{W}

\renewcommand{\leq}{\leqslant} \renewcommand{\geq}{\geqslant}

\providecommand{\ceil}[1]{\lceil#1\rceil}
\providecommand{\Abs}[1]{\left|#1\right|}

\DeclarePairedDelimiter{\abs}{\lvert}{\rvert}
\DeclarePairedDelimiter{\floor}{\lfloor}{\rfloor}
\DeclarePairedDelimiter{\set}{\lbrace}{\rbrace}
\DeclarePairedDelimiter{\parens}{\lparen}{\rparen}
\DeclarePairedDelimiter{\brackets}{\lbrack}{\rbrack}

\def\real{{x}}

\def\1int{{[0,1]}}

\title[Khintchine's Theorem with random fractions]{Khintchine's
  Theorem with random fractions} 

\author[F.~A.~Ram{\'i}rez]{Felipe
  A.~Ram{\'i}rez} 


\address{Wesleyan University,
  Middletown CT, USA} 

\email{framirez@wesleyan.edu}

\dedicatory{For Luna Luc{\'i}a.}

\begin{document}

% ===============================================

\begin{abstract}
  We prove versions of Khintchine's Theorem (1924) for approximations
  by rational numbers whose numerators lie in randomly chosen sets of
  integers, and we explore the extent to which the monotonicity
  assumption can be removed. Roughly speaking, we show that if the
  number of available fractions for each denominator grows too fast,
  then the monotonicity assumption cannot be removed. There are
  questions in this random setting which may be seen as cognates of
  the Duffin--Schaeffer Conjecture (1941), and are likely to be more
  accessible. We point out that the direct random analogue of the
  Duffin--Schaeffer Conjecture, like the Duffin--Schaeffer Conjecture
  itself, implies Catlin's Conjecture (1976). It is not obvious
  whether the Duffin--Schaeffer Conjecture and its random version
  imply one another, and it is not known whether Catlin's Conjecture
  implies either of them. The question of whether Catlin implies
  Duffin--Schaeffer has been unsettled for decades.
\end{abstract}

% ===============================================

\maketitle
\thispagestyle{empty}

\setcounter{tocdepth}{1} %excludes subsections from toc
\tableofcontents
% {\footnotesize{\tableofcontents}}

\section{Introduction and results}
\label{sec:altint}

Let $\Psi:\NN\to\RR_{\geq 0}$ be some function, and for $n\in\NN$,
let $[n]:=\{1, \dots, n\}$. In \emph{metric Diophantine approximation}
we are concerned with properties of the set
\begin{equation*}
  \calW(\Psi) := \set*{\real\in [0,1] : \abs*{\real - \frac{a}{n}}<\Psi(n) \textrm{ for infinitely many } n\in\NN, a\in [n]},
\end{equation*}
the \emph{$\Psi$-approximable numbers}. One of the foundational results
in this area is Khintchine's Theorem.

\theoremstyle{plain} \newtheorem*{kt}{Khintchine's Theorem}
\begin{kt}[\cite{Khintchineonedimensional}, 1924]
  If $\Psi : \NN \to \RR_{\geq 0}$ is non-increasing, then
  \begin{equation*}
    \lambda\parens*{\calW (\Psi)} = 
    \begin{cases}
      1 &\textrm{if }\sum_{n=1}^\infty n\Psi(n) =\infty \\
      0 &\textrm{if }\sum_{n=1}^\infty n\Psi(n) < \infty
    \end{cases}
  \end{equation*}
  where $\lambda$ is Lebesgue measure.
\end{kt}

\begin{remark}
  Actually, Khintchine originally proved this with the stronger
  assumption that $n^2 \Psi(n)$ is non-increasing. 
\end{remark}

It is a fact, proved by Duffin and Schaeffer in 1941, that
Khintchine's Theorem becomes false if one removes the monotonicity
assumption~\cite{duffinschaeffer}. However, there is a long-standing
conjecture that if we only allow approximations by reduced fractions,
then a modified version of Khintchine's Theorem holds without the need
for monotonicity.  For each $n\in\NN$, let $[n]'$ denote the elements
of $[n]$ which are co-prime to $n$, and consider
\begin{equation*}
  \calW'(\Psi) := \left\{\real\in [0,1] : \Abs{\real - \frac{a}{n}}<\Psi(n) \textrm{ for infinitely many } n\in\NN,  a\in[n]'\right\},
\end{equation*}
the numbers that are \emph{$\Psi$-approximable by reduced
  fractions}. Here is the conjecture.

\theoremstyle{plain} \newtheorem*{dsconj}{Duffin--Schaeffer
  Conjecture}
\begin{dsconj}[\cite{duffinschaeffer}, 1941]
  For any $\Psi:\NN\to \RR_{\geq 0}$,
  \begin{equation*}
    \lambda\parens*{\calW' (\Psi)} = 
    \begin{cases}
      1 &\textrm{if }\sum_{n=1}^\infty \varphi(n)\Psi(n) =\infty \\
      0 &\textrm{if }\sum_{n=1}^\infty \varphi(n)\Psi(n) < \infty
    \end{cases}
  \end{equation*}
  where $\varphi(n) := \#[n]'$ is Euler's totient function.
\end{dsconj}

The Duffin--Schaeffer Conjecture remains open, and is one of the most
pursued problems in metric number theory. The difficulty in it arises
from the fact that the sets $W'(\Psi)$---and also $W(\Psi)$---are
$\limsup$s of sequences of sets that are not in general
independent. But our only tools for asserting full measure of
$\limsup$ sets are partial converses to the Borel--Cantelli Lemma,
which all work by verifying some suitably weakened version of
independence of sets. Indeed, all the progress that has been made on
(and related to) the Duffin--Schaeffer Conjecture has revolved around
the issue of
independence~\cite{AistleitneretalExtraDivergence,AistleitnerDS,BHHVextraii,duffinschaeffer,ErdosDS,HaynespadicDSC,HPVextra,PollingtonVaughan,Vaaler}.

Motivated by this state of affairs, we seek to study approximations by
fractions whose ``permitted'' numerators lie in randomly chosen subsets
of $[n]$, with the hope that this extra randomness will afford us
enough probabilistic independence to show that our $\limsup$ sets have
full measure. The main results of this paper are versions of
Khintchine's Theorem which hold for approximations by random
fractions. For example:

\begin{theorem}\label{thm:totallyrandom}
  For every $n\in\NN$ pick a subset $P_n\subset[n]$ uniformly at
  random, and denote the resulting sequence of subsets by
  $P = (P_n)_{n=1}^\infty$.  Let
  \begin{equation*}
    \calW^\calP(\Psi) := \set*{\real \in\1int : \Abs{\real - \frac{a}{n}}<\Psi(n) \textrm{ for infinitely many } n\in \NN, a\in\calP_n }.
  \end{equation*}
  Then the following almost surely holds: For any nonincreasing
  $\Psi:\NN\to\RR_{\geq 0}$, we have
  \begin{equation*}
    \lambda \parens*{W^\calP(\Psi)} =
    \begin{cases}
      1 &\textrm{if $\sum_{n=1}^\infty n\Psi(n)=\infty$,} \\
      0 &\textrm{if $\sum_{n=1}^\infty n\Psi(n)< \infty$.}
    \end{cases}
  \end{equation*}
\end{theorem}

In words: if for each denominator we randomly and uniformly delete a
subset of the possible numerators, then with full probability,
Khintchine's Theorem will continue to hold with the rationals that are
left over. This is actually a consequence of our next
result, which holds in a more challenging setting.

Fix a sequence $f:\NN\to\ZZ$ such that $0\leq f(n) \leq n$
for all $n$. Suppose we declare that for each denominator $n\in\NN$,
we will use exactly $f(n)$ of the $n$ possible numerators from
$[n]$. That is, we choose some
\begin{equation*}
  P_n\in \Omega_{f,n}:=\set*{Q\subset [n] : \#Q = f(n)}
\end{equation*}
according to the uniform probability measure $\PP_{f,n}$ on the collection
$\Omega_{f,n}$ of $f(n)$-element subsets of $[n]$. This results in an
element $P = (P_n)_{n=1}^\infty$ of the product space
\begin{equation*}
  \Omega_f := \prod_{n=1}^\infty \Omega_{f,n}
\end{equation*}
which has been chosen according to the product measure
$\PP_f= \bigotimes_{n=1}^\infty \PP_{f,n}$. 

The following theorem is the main result of this paper.

\begin{theorem}\label{thm:boundedandlinear}
  If the average order of $f$ is positive and bounded, or if it is at
  least linear in $n$, then for $\PP_f$-almost every
  $\calP\in\Omega_f$, the following holds:
  \begin{equation}\label{eq:rkt}
    \begin{split}
      \textrm{For any decreasing $\Psi:\NN\to\RR_{\geq 0}$, we have}\\
      \lambda \parens*{W^\calP(\Psi)} =
      \begin{cases}
        1 &\textrm{if $\sum_{n=1}^\infty f(n)\Psi(n)=\infty$,} \\
        0 &\textrm{if $\sum_{n=1}^\infty f(n)\Psi(n)< \infty$.}
      \end{cases}
    \end{split}
  \end{equation}
\end{theorem}

\begin{remark}
  The theorem applies, in particular, when $f\equiv \varphi$.  The
  average order of $\varphi(n)$ is $6n/\pi^2$ \cite[Theorem
  330]{HardyWright}. It also applies in the case that $f(n)=n$,
  where it gives Khintchine's Theorem.
\end{remark}

\begin{remark}
  The role that the summand $f(n)\Psi(n)$ plays in
  Theorem~\ref{thm:boundedandlinear} is the same as the role of
  $n\Psi(n)$ in Khintchine's Theorem, and $\varphi(n)\Psi(n)$ in the
  Duffin--Schaeffer Conjecture. One should think of it as (half) the
  measure of the set of points in $[0,1]$ that are within $\Psi(n)$ of
  one of the permitted rationals of denominator $n$. 
\end{remark}

When the average order of $f$ is linear in $n$, we get the
following corollary to Theorem~\ref{thm:boundedandlinear}. 

\begin{corollary}\label{cor:philike}
  If the average order of $f$ is at least linear in $n$, then for
  $\PP_f$-almost every $\calP\in\Omega_f$, the following holds: given
  any decreasing $\Psi:\NN\to\RR_{\geq 0}$, we have
  \begin{equation*}
    \lambda \parens*{W^\calP(\Psi)} =
    \begin{cases}
      1 &\textrm{if $\sum_{n=1}^\infty n\Psi(n)=\infty$,} \\
      0 &\textrm{if $\sum_{n=1}^\infty n\Psi(n)< \infty$.}
    \end{cases}
  \end{equation*}
\end{corollary}

When the average order of $f$ is positive and bounded, we get the following
corollary to Theorem~\ref{thm:boundedandlinear}.

\begin{corollary}\label{cor:boundedlike}
  If the average order of $f$ is positive and bounded, then for
  $\PP_f$-almost every $\calP\in\Omega_f$, the following holds: given
  any decreasing $\Psi:\NN\to\RR_{\geq 0}$, we have
  \begin{equation*}
    \lambda \parens*{W^\calP(\Psi)} =
    \begin{cases}
      1 &\textrm{if $\sum_{n=1}^\infty \Psi(n)=\infty$,} \\
      0 &\textrm{if $\sum_{n=1}^\infty \Psi(n)< \infty$.}
    \end{cases}
  \end{equation*}
\end{corollary}

These corollaries are immediate after noticing that $\sum f(n)\Psi(n)$
is equivalent to $\sum \Psi(n)$ when $f$ is positive and bounded on
average, and to $\sum n\Psi(n)$ when it is linear on average.

\

It is unclear whether Theorem~\ref{thm:boundedandlinear} should truly
require the restrictions on the average order of $f$, so we ask the following.

\begin{question}\label{conj:randomkt}
  Given an arbitrary sequence $f(n)$ with $0\leq f(n)\leq n$, does
  property~(\ref{eq:rkt}) hold $\PP_f$-almost surely?
\end{question}

Theorem~\ref{thm:boundedandlinear} can be stated as ``the answer
to Question~\ref{conj:randomkt} is `yes' for sequences that are either
positive and bounded on average, or linear on average.'' We do not
pursue Question~\ref{conj:randomkt} any further here. Let us only
offer a family of functions $f$ for which it seems our methods might
be adaptable, namely, the family of $f$ for which one can find a
lacunary sequence $\set{N_t}\subset \NN$ with the following two
properties holding for all large $t$:
\begin{gather}
  A \leq \frac{F_{t+1}}{F_t} \leq B \label{eq:blocksumgrowth}\\
  F_t \asymp N_t f(N_t) \quad(t\to\infty), \label{eq:averages}
\end{gather}
where $A$ and $B$ are constants such that $1 < A \leq B < \infty$, and
$F_t = \sum_{n\in(N_t, N_{t+1}]}f(n)$ are block-sums.

\

Returning to the original motivation for this work, consider the
direct random analogue of the Duffin--Schaeffer Conjecture:

\begin{conjecture}[Random Duffin--Schaeffer Conjecture]\label{conj:randomdsc}
  $\PP_\varphi$-almost surely,~(\ref{eq:rkt}) holds without the
  monotonicity assumption.
\end{conjecture}

While the Duffin--Schaeffer Conjecture is a statement about a specific
point in $\Omega_\varphi$, the Random Duffin--Schaeffer Conjecture
asks about the \emph{generic} point in $\Omega_\varphi$. There seems
to be no obvious implication between the two statements, one way or
the other (although Conjecture~\ref{conj:randomdsc} can be worded
appealingly, if unrigorously, as \emph{``the Duffin--Schaeffer
  Conjecture is almost surely true''}), so we cannot claim that
Conjecture~\ref{conj:randomdsc} would represent direct progress toward a
proof of the Duffin--Schaeffer Conjecture. However,
in~\S\ref{sec:wdsc=catlin} we show that
Conjecture~\ref{conj:randomdsc} does imply the following well-known
open problem.

\theoremstyle{plain}
\newtheorem*{catconj}{Catlin's Conjecture}
\begin{catconj}[\cite{CatlinConj}, 1976]
  For any $\Psi:\NN\to \RR_{\geq 0}$,
  \begin{equation*}
    \lambda\parens*{\calW (\Psi)} = 
    \begin{cases}
      1 &\textrm{if } \sum_{n=1}^\infty \varphi(n)\max\set*{\Psi(kn) : k\in\NN} =\infty \\
      0 &\textrm{if }\sum_{n=1}^\infty
      \varphi(n)\max \set*{\Psi(kn) : k\in\NN} < \infty.
    \end{cases}
  \end{equation*}
\end{catconj}

It is not known whether Catlin's Conjecture implies the
Duffin--Schaeffer Conjecture. (It was claimed so in~\cite{CatlinConj},
but an error was found by Vaaler~\cite{Vaaler}, and the issue remains
unresolved.) Regardless, one might expect that even if Catlin's
Conjecture is weaker than the Duffin--Schaeffer Conjecture, it may
still be very difficult to prove, and therefore
Conjecture~\ref{conj:randomdsc} might also be comparably
difficult. Indeed, we have not managed to remove the monotonicity
assumption from Theorem~\ref{thm:boundedandlinear}, even for the
simplest case $f\equiv 1$. We therefore ask the following question.

\begin{question}\label{question:which}
  Are there sequences $f(n)$ with $0\leq f(n) \leq n$ for which,
  $\PP_f$-almost surely,~(\ref{eq:rkt}) holds without the monotonicity
  assumption? If so, which? 
\end{question}

Clearly, Conjecture~\ref{conj:randomdsc} answers the first part of
this question affirmatively. But a general answer may be of
independent interest, and a partial answer (such as treating bounded
sequences) is likely more accessible than
Conjecture~\ref{conj:randomdsc}.

\

We explore cases where the answer to Question~\ref{question:which} is
``no.'' It is not hard to see that the same counterexample provided by
Duffin and Schaeffer prevents us from removing monotonicity from the
statement of Theorem~\ref{thm:boundedandlinear} whenever $f(n)\gg n$.
We give a counterexample whose rate of divergence we can track, and
which immediately shows that monotonicity cannot be removed if $f(n)$
grows faster than $n/\log\log n$.

\begin{theorem}[Duffin--Schaeffer-style
  Counterexample]\label{thm:counterex}
  Suppose that
  \begin{equation*}
    \lim_{n\to\infty}\frac{f(n)\log\log n}{n} = \infty.
  \end{equation*}
  Then there exists a function $\Psi:\NN\to\RR_{\geq 0}$ such
  that $\sum_{n=1}^\infty f(n)\Psi(n)$ diverges, yet
  $\lambda\parens*{\calW(\Psi)}=0$.
\end{theorem}

Since for any $\calP\in\Omega_f$ and $\Psi$, we have
$\calW^\calP(\Psi)\subset\calW(\Psi)$, it is clear that
$\calW^\calP(\Psi)$ is null whenever $\calW(\Psi)$ is. Therefore,
Theorem~\ref{thm:counterex} gives us the following.

\begin{corollary}\label{cor:monotonicity}
  If $\parens*{f(n)\log\log n}/n \to\infty$, then there exist
  functions $\Psi$ such that $\sum_nf(n)\Psi(n)$ diverges, yet
  $\lambda\parens*{\calW^\calP(\Psi)}=0$ for \emph{every} $\calP\in\Omega_f$.
\end{corollary}

Notice that $(\varphi(n)\log\log n)/n$ can be made arbitrarily large,
for example by taking prime values of $n$. On the other hand, it is
well-known that
\begin{equation}\label{eq:upperboundonphi}
  \varphi(n) < \frac{n}{e^\gamma \log\log n}
\end{equation}
for infinitely may $n$. One might make the (overly optimistic) guess
that if a function $f$ has the following two things in common with
$\varphi$:
\begin{itemize}
\item $(f(n)\log\log n)/n \to \infty$ along a subsequence of $\NN$,
  and
\item $(f(n)\log\log n)/n$ is bounded on some other
  subsequence,
\end{itemize}
then this is enough to guarantee that we will never find a function
$\psi$ like the one in Theorem~\ref{thm:counterex}. But the next
corollary shows that that is not necessarily true. The divergence of
$(f(n)\log\log n)/n$ in Corollary~\ref{cor:monotonicity} can occur
along a subsequence, while keeping the expression bounded off of that
subsequence.

\begin{corollary}\label{cor:sweetspot}
  There exist sequences $1\leq f(n)\leq n$ such that
  \begin{equation*}
    \begin{dcases}
      \lim_{i\to\infty}\frac{f(k_i)\log\log k_i}{k_i} = \infty &\textrm{for some subsequence } \set{k_i}\subsetneq \NN, \textrm{ and} \\
      \frac{f(n)\log\log n}{n} &\textrm{is bounded uniformly over all
      } n\notin\set{k_i},
    \end{dcases}
  \end{equation*}
  and such that there exists a function $\Psi:\NN\to\RR_{\geq 0}$
  such that $\sum_{n}f(n)\Psi(n)$ diverges, yet
  $\lambda\parens*{\calW(\Psi)}=0$, hence
  $\lambda\parens*{\calW^\calP(\Psi)}=0$ for every $\calP\in\Omega_f$.
\end{corollary}

\section{Notation}
\label{sec:notation}

For two functions $f$ and $g$, we use $f\ll g$ to mean that there exists
some constant $C>0$ such that $f \leq Cg$. This is synonymous with
$g \gg f$. We use $f\asymp g$ to mean that $f\ll g$ and $f \gg g$ both
hold. Usually, it is specified that we only need these relations to hold
for all large arguments of the functions. 

\section{Proof of Theorem~\ref{thm:totallyrandom}}
\label{sec:totallyrandomproof}

In this section we give a proof of Theorem~\ref{thm:totallyrandom}
which relies on Theorem~\ref{thm:boundedandlinear}. The main idea lies
in the fact that, with probability $1$, one picks subsets of $[n]$
whose cardinalities grow linearly on average.

Let $\Omega$ denote the space of sequences $P = (P_n)_{n=1}^\infty$ of
subsets $P_n\subset[n]$, and denote by $\PP$ the probability measure
described in Theorem~\ref{thm:totallyrandom}. Note that $\Omega$ is
the disjoint union
\begin{equation*}
  \Omega = \bigcup_{f\in\calF} \Omega_f
\end{equation*}
where $\calF$ denotes the collection of all possible sequences
$0\leq f(n)\leq n$. In turn, we may see $\calF$ as the sequence space
\begin{equation*}
\calF = \prod_{n=1}^\infty \calF_n 
\end{equation*}
where $\calF_n = \set*{0,1, \dots, n}$. The space $\calF$ carries the
measure $\nu = \bigotimes_{n=1}^\infty \nu_n$, where $\nu_n$ is the
probability measure on $\calF_n$ defined by
$\nu_n(\set{k}) = \binom{n}{k}\parens*{\frac{1}{2}}^n$. Notice that
$\nu$ is the pushforward $\pi_* (\PP)$ of $\PP$ under the natural
projection $\pi:\Omega \to \calF$. Then $\PP$ disintegrates as
\begin{equation*}
  \PP = \int_\calF \PP_f\,d\nu(f).
\end{equation*}
In other words, the random process described in
Theorem~\ref{thm:totallyrandom} is the same as the following: first,
for each $n$ choose $f(n)$ according to $\nu_n$, and then choose
$P\in\Omega_f$ according to $\PP_f$.

\begin{lemma}\label{lem:P(L)=1}
The average order of $\nu$-almost every $f\in\calF$ is at least linear.
\end{lemma}

\begin{proof}
  Treating $f(n)$ as a random variable on $(\calF_n, \nu_n)$, we have
  the expectation $\EE(f(n))=\frac{n}{2}$ and variance
  $\sigma^2(f(n)) = \frac{n}{4}$. In particular, for every $N\in \NN$,
  we have $\EE\parens*{X_N} \asymp N^2$ and
  $\sigma^2\parens*{X_N} \asymp N^2$, where
  $X_N := \sum_{n=1}^N f(n)$. Chebyshev's inequality says that for any
  $y>0$, we have
  \begin{equation*}
    \nu\parens*{\abs*{ X_N - \EE(X_N)} > y} \leq \frac{\sigma^2(X_N)}{y^2}.
  \end{equation*}
  Setting $y = \EE(X_N)/2$, see that
  \begin{equation*}
    \nu\parens*{X_N < \frac{1}{2}\EE(X_N)} \leq 4\frac{\sigma^2(X_N)}{\EE(X_N)^2} \ll \frac{1}{N^2}.
  \end{equation*}
  Since $\sum_N N^{-2}$ converges, we find by the Borel--Cantelli
  Lemma that the probability is zero that $X_N < \frac{1}{2}\EE(X_N)$
  for infinitely many $N$. In other words, $\nu$-almost every
  $f\in \calF$ grows at least linearly on average.
\end{proof}

\begin{proof}[Proof of Theorem~\ref{thm:totallyrandom}] 
  Let $K$ denote the elements of $\Omega$ for which~(\ref{eq:rkt})
  holds, and let $L \subset \calF$ be the collection of sequences in
  $\calF$ whose average order is at least linear in $n$. Then
  \begin{align*}
    \PP(K) &= \int_\calF \PP_{f}(K)\,d\nu(f) \\
&\geq \int_L \PP_f(K)\,d\nu(f) \\
&\overset{\textrm{Thm.~\ref{thm:boundedandlinear}}}{=} \int_L 1\,d\nu(f)\\
&=\nu(L) \overset{\textrm{Lem.~\ref{lem:P(L)=1}}}{=} 1. 
  \end{align*}
  This proves the theorem, after observing that for any $f\in L$ and
  monotonic $\Psi$, the series $\sum f(n)\Psi(n)$ and
  $\sum n\Psi(n)$ either both converge or both diverge.
\end{proof}

\section{Ubiquitous systems}
\label{sec:ubiquity}

We prove Theorem~\ref{thm:boundedandlinear} using the tools of
ubiquity theory. For a detailed guide through this theory, we refer
to~\cite{BDV}, and for a more survey-like treatment,
see~\cite{beresnevich_ramirez_velani_2016}. Here are the main points.

\begin{definition*}[Local ubiquity]
  Let $\mathcal R \subseteq [0,1]\cap\QQ$ be a set of rational points,
  indexed by some subset $\mathcal J \subseteq \set{(a,n) : n\in\NN, a \in [n]}$ by the
  mapping $(a,n) \mapsto \frac{a}{n}$.  Suppose there exists a
  function $\rho:\RR^+\to\RR^+$ with $\lim_{r\to\infty}\rho(r) = 0$, a
  constant $\kappa>0$, and sequences $M_t < N_t$
  ($M_t\uparrow\infty$), such that for any sufficiently small interval
  $I\subset [0,1]$, there exists $t_0:=t_0(I)$ such that
  \begin{equation*}
    \lambda\parens*{I\cap \bigcup_{\substack{(a,n)\in\mathcal J : \\ M_t < n \leq N_t}} B\parens*{\frac{a}{n}, \rho(N_t)}} \geq \kappa \lambda(I)
  \end{equation*}
  for all $t\geq t_0$. In this case we say that $(\mathcal R,\mathcal J)$ is a
  \emph{locally ubiquitous system relative to}
  $(\rho, \set{M_t}, \set{N_t})$.
\end{definition*}

For $\Psi:\RR^+\to\RR^+$ a decreasing function, 
let
\begin{equation*}
  \Lambda(\mathcal R, \mathcal J, \Psi) = \set*{x\in[0,1] : x\in
    B\parens*{\frac{a}{n}, \Psi(n)} \textrm{ for infinitely many
      $(a,n)\in\mathcal J$}}.
\end{equation*}
For example, if $\mathcal R = [0,1]\cap \QQ$ and
$\mathcal J = \set{(a,n) : n\in\NN, a\in[n]}$, then
$\Lambda(\mathcal R, \mathcal J, \Psi) = W(\Psi)$. If for some
$P\in\Omega_f$, we put

\begin{equation*}
  \QQ(P_n) := \set*{\frac{a}{n} : a\in P_n},
\end{equation*}
and set
\begin{equation*}
  \mathcal R = \QQ(P_n) := \set*{\frac{a}{n} : a\in P_n} \qquad\textrm{and}\qquad   \mathcal J = J(P) := \set*{(a,n) : a\in P_n},
\end{equation*}
then $\Lambda(\mathcal R, \mathcal J, \Psi) = W^P(\Psi)$.

The following is a special case of~\cite[Theorem~1, Corollary~2]{BDV}. 

\begin{theorem}[Ubiquity theorem]\label{thm:ubiquity}
  Suppose $(\mathcal R, \mathcal J)$ is a locally ubiquitous system
  relative to $(\rho, \set{M_t}, \set{N_t})$ and that
  $\Psi:\RR^+\to\RR^+$ is a decreasing function. If $\rho$ is
  $\set{N_t}$-regular, and $\sum\frac{\Psi(N_t)}{\rho(N_t)} = \infty$,
  then $\lambda(\Lambda(\mathcal R,\mathcal J, \Psi))=1$.
\end{theorem}

Our strategy for proving the divergence part of
Theorem~\ref{thm:boundedandlinear} is to establish that
$(\QQ(P), J(P))$ is almost surely a ubiquitous system relative to a
suitable ubiquitous function $\rho$ and sequence $\set{N_t}$, and then
to apply Theorem~\ref{thm:ubiquity}. (We will use $M_t = N_{t-1}$.)

\section{Random local ubiquity when $f(n)$ is linear on average}
\label{sec:phiubiq}

The purpose of this section is to show that if the average order of
$f(n)$ is at least linear, then for $\PP_f$-almost every
$P\in\Omega_f$, the pair $(\QQ(P), J(P))$ are a locally ubiquitous
system. 

First, we prove the following equivalent for average order being at
least linear. This lemma provides the sequences relative to which we
will prove local ubiquity.

\begin{lemma}\label{lem:averageorder}
  Let $0\leq f(n) \leq n$ be an integer sequence. Its average order is
  at least linear if and only if there exists an integer sequence
  $\set{N_t = k^t}$, where $k>1$, such that the block-sums
  \begin{equation*}
    F_t := \sum_{n\in (N_t,N_{t+1}]} f(n)
  \end{equation*}
  satisfy $F_t \gg N_{t+1}^2$ for all sufficiently large
  $t$. Furthermore, one can choose $k$ large enough to ensure that
  there exists some positive $\lambda < 1$ such that
  $F_t \leq \lambda F_{t+1}$ for all sufficiently large $t$.
\end{lemma}

\begin{proof}
  First, suppose that $a>0$ is such that
  \begin{equation*}
    \sum_{n=1}^N f(n) \geq \sum_{n=1}^N an = \frac{a}{2}(N^2 + N)
  \end{equation*}
  holds for all sufficiently large $N$. Let $N_t = k^t$, with $k>1$ to
  be chosen later. Then we have 
  \begin{equation*}
    \sum_{n=N_t+1}^{N_{t+1}} f(n) \geq \frac{a}{2}(N_{t+1}^2 + N_{t+1}) - \frac{1}{2}(N_t^2 + N_t) \geq \parens*{\frac{a}{2} - \frac{1}{2k^2}}N_{t+1}^2
  \end{equation*}
  for all sufficiently large $t$. Now choose $k\in\NN$ large enough
  that the parenthetical factor exceeds, say, $a/4$. This proves one
  direction of the ``if and only if'' statement.

  We can prove the ``furthermore'' statement by choosing $k$ larger
  still, if need be. Notice that
  \begin{equation*}
    F_t \leq N_{t+1}^2 = \frac{1}{k^2} N_{t+2}^2 \leq \frac{a}{4k^2}F_{t+1},
  \end{equation*}
  holds for all sufficiently large $t$, by what we have just
  proved. The ``furthermore'' statement only requires us to have
  chosen $k$ large enough that $\frac{a}{4k^2}<1$.

  We now prove the other direction of the ``if and only if.'' Suppose
  that there is an integer sequence $N_t = k^t$ such that
  $F_t \gg N_{t+1}^2$ for all sufficiently large $t$. Let
  $N_t \leq N < N_{t+1}$. Then
  \begin{equation*}
    \sum_{n=1}^N f(n) \geq \sum_{n=1}^{N_t} f(n) \gg N_t^2 \gg N^2 \gg \sum_{n=1}^N n.
  \end{equation*}
  The constant $a$ is absorbed into the the notation $\gg$.
\end{proof}

Let
\begin{equation*}
  Q_n := \set*{\frac{a}{n}\in[0,1] : (a,n)=1}, 
\end{equation*}
that is, $Q_n$ is the set of fractions in $[0,1]$ which in reduced
form have denominator $n$. The next lemma is a simple and well-known
fact, which we prove following~\cite[Lemma~1]{NiederreiterFarey}. 

\begin{lemma}\label{lem:niederreiter}
  For any non-tivial interval $I\subset[0,1]$, there exists some
  integer $n_0(I)$ such that
  $\#(Q_n\cap I) \geq \frac{1}{2}\varphi(n)\lambda(I)$ whenever
  $n\geq n_0(I)$.
\end{lemma}

\begin{proof}
  Let $\theta_I(n) = \#\set*{a/n\in I}$, and notice that
  $\floor{\lambda(I)n} \leq \theta_I(n)\leq \floor{\lambda(I)n}+1$. We have
  \begin{equation*}
    \theta_I(n) = \sum_{d\mid n} \#(Q_d\cap I),
  \end{equation*}
  and the M{\"o}bius inversion formula gives
  \begin{align*}
    \#(Q_n\cap I) &= \sum_{d\mid n} \mu\parens*{\frac{n}{d}}\theta_I(d)\\
                  &\geq \sum_{d\mid n} \mu\parens*{\frac{n}{d}}\lambda(I)d - \sum_{d\mid n} \mu\parens*{\frac{n}{d}}\set{\lambda(I)d} \\
                  &= \varphi(n)\lambda(I)- \sum_{d\mid n} \mu\parens*{\frac{n}{d}}\set{\lambda(I)d}\\
                  &= \varphi(n)\lambda(I) + O(\sqrt{n}). 
  \end{align*}
  The lemma follows, since $\sqrt{n}/\varphi(n)\to 0, (n\to\infty)$.
\end{proof}

The next lemma is a simple geometric result which is only used in the
lemma that follows it.

\begin{lemma}\label{lem:normcomparison}
  Let $c_1, c_2>0$. For any $x = (x_1, \dots, x_d)\in \ZZ^{d}$ such
  that $x_i^2 \leq c_1 d, (i=1, \dots, d)$ and
  $\abs{x}_2\geq d\sqrt{c_2}$, we have
  $\abs{x}_1\gg \sqrt{d}\, \abs{x}_2$, with an implied constant only
  depending on $c_1$ and $c_2$.
\end{lemma}

\begin{proof}
  On one hand we have $x_1^2 + \dots + x_d^2 \geq c_2 d^2$, and on the
  other, we have that $x_i^2 \leq c_1 d$ for $i=1, \dots, d$.
  Together, these facts imply that the number of non-zero entries of
  $x = (x_1, \dots, x_d)$ is at least $(c_2/c_1)d$.

  Now, since balls in the $\abs{\cdot}_1$-norm achieve their maximal
  $\abs{\cdot}_2$-norms on the coordinate axes, the ratio
  $\frac{\abs{x}_1}{\abs{x}_2}$ will be minimized when all but
  $\ceil{(c_2/c_1)d}$ coordinates of $x$ are $0$,
  $\ceil{(c_2/c_1)d}-1$ are $1$, and the remaining coordinate is
  $\sqrt{c_1d}$. For such an $x$, we will have
  \begin{equation*}
    \frac{\abs{x}_1^2}{\abs{x}_2^2} = \frac{(\ceil{(c_2/c_1) d} - 1+ \sqrt{c_1d})^2}{(\ceil{(c_2/c_1) d} - 1+ c_1d)} \gg d,
  \end{equation*}
  which proves the claim.
\end{proof}

\begin{lemma}\label{lem:sumfphi}
  Suppose $0\leq f(n) \leq n$ for all $n$, and $\set{N_t = k^t}$ is
  a geometric sequence such that  
  \begin{equation*}
    \sum_{n\in(N_t, N_{t+1}]}f(n) \gg N_{t+1}^2
  \end{equation*}
  for sufficiently large $t$. Then
  \begin{equation*}
    \sum_{n\in(N_t, N_{t+1}]} \frac{f(n)\varphi(n)}{n} \gg \sum_{n\in(N_t, N_{t+1}]} f(n)
  \end{equation*}
  for large $t$.
\end{lemma}

\begin{remark}
  Before stating the proof, we mention that the converse fails. For
  example, if $f$ is supported on $\set{N_t}$, then the conclusion
  holds, yet the assumption does not. No matter. We state and prove it
  in the context where we will need it.
\end{remark}

\begin{proof}
  By the reverse H{\"o}lder inequality, we have
  \begin{equation*}
    \sum_{n\in(N_t, N_{t+1}]} \frac{f(n)\varphi(n)}{n} \geq \parens*{\sum_{n\in(N_t, N_{t+1}]} \parens{f(n)}^{\frac{1}{2}}}^2\parens*{\sum_{n\in(N_t, N_{t+1}]} \parens*{\frac{\varphi(n)}{n}}^{-1}}^{-1}.
  \end{equation*}
  We have $N_{t+1} = \frac{k}{k-1}(N_{t+1}-N_t)$ for all $t$, so the
  first factor above can be handled by Lemma~\ref{lem:normcomparison},
  with
  \begin{equation*}
    x=(f(N_t+1)^{1/2}, \dots, f(N_{t+1})^{1/2})
  \end{equation*}
  and $d = N_{t+1}-N_t$.  We get
\begin{equation}
  \sum_{n\in(N_t, N_{t+1}]} \frac{f(n)\varphi(n)}{n} \gg (N_{t+1}-N_t) \parens*{\sum_{n\in(N_t, N_{t+1}]} \parens{f(n)}}\parens*{\sum_{n\in(N_t, N_{t+1}]} \parens*{\frac{\varphi(n)}{n}}^{-1}}^{-1} \label{eq:comeback}
  \end{equation}
  And, it is a fact (found, for example, in~\cite[Lemma~2.5]{Harman})
  that
\begin{equation*}
  \sum_{n\in(0,N]} \parens*{\frac{n}{\varphi(n)}} = \frac{315\zeta(3)}{2\pi^4} N + O(\log N).
\end{equation*}
  In particular, 
  \begin{equation}\label{eq:n/phi}
    \sum_{n\in(N_t,N_{t+1}]} \parens*{\frac{n}{\varphi(n)}} = \frac{315\zeta(3)}{2\pi^4} (N_{t+1}-N_t) + O(\log N_{t+1}).
  \end{equation}
  The lemma follows on combining~(\ref{eq:comeback})
  and~(\ref{eq:n/phi}).
\end{proof}

\begin{lemma}\label{lem:fareycoupons}
  Let $0\leq f(n)\leq n$ be a sequence which is at least linear on
  average, and let $\set{N_t = k^t}$ be a geometric sequence as in
  Lemma~\ref{lem:averageorder}. For each $n\in (N_t, N_{t+1}]$,
  uniformly and independently choose an $f(n)$-element subset
  $P_n \subset [n]$. For an interval $I\subset [0,1]$,
  let
  \begin{equation*}
    Y_n(I) = \#(\QQ(P_n)\cap Q_n \cap I)\qquad\textrm{and}\qquad X_t(I) = \sum_{n\in(N_t, N_{t+1}]} Y_n(I). 
  \end{equation*}
  Then there exist constants $C_1, C_2 >0$ depending only on $f$ and
  $\set{N_t}$ such that for any interval $I\subset [0,1]$ there exists
  $t_0:=t_0(I)$, such that
  \begin{align}
    \EE(X_t(I)) &\geq C_1 F_t \lambda(I) \label{eq:expectation}
\intertext{and}
\sigma^2 (X_t(I)) &\leq C_2 F_t  \label{eq:variance}
  \end{align}
  whenever $t\geq t_0$.
\end{lemma}

\begin{proof}
  The lemma is essentially a classic urn problem from probability
  theory. There are $n$ marbles in an urn, of which $\#(Q_n\cap I)$
  are distinquished. It is well-known that if we draw $f(n)$ marbles
  from this urn, then the expected number of distinguished marbles
  drawn is
  \begin{equation}\label{eq:exp}
    \frac{f(n) \#(Q_n\cap I)}{n},
  \end{equation}
  and the variance is bounded above by
  \begin{equation}\label{eq:var}
    \frac{f(n) \#(Q_n\cap I) (n-\#(Q_n\cap I))}{n^2}.
  \end{equation}
  (See, for example,~\cite[Example IX.5(d)]{Feller}.)

  By Lemma~\ref{lem:niederreiter}, we have
  $\#(Q_n\cap I) \geq \frac{1}{2}\varphi(n)\lambda(I)$ for all
  $n\geq n_0(I)$.  Combining with~(\ref{eq:exp}), we find
  $\EE(Y_n(I))\geq \frac{f(n)\varphi(n)}{2n}\lambda(I)$ for all
  $n\geq n_0 (I)$.  Therefore,
  \begin{equation*}
    \EE(X_t(I)) \geq \sum_{n\in(N_t,N_{t+1}]}\frac{f(n)\varphi(n)}{2n}\lambda(I),
  \end{equation*}
  for large $t$. Finally, we appeal to Lemma~\ref{lem:sumfphi} to
  finish the proof of~(\ref{eq:expectation}).

  For the variance note that~(\ref{eq:var}) is bounded above by
  $\frac{f(n)}{4}$, so that
  $\sigma^2(Y_n(I))\leq \frac{f(n)}{4}$. It follows that
  \begin{equation*}
    \sigma^2(X_t(I)) \leq \frac{1}{4} \sum_{n\in(N_t, N_{t+1}]}f(n),
  \end{equation*}
  which proves~(\ref{eq:variance}) with $C_2 = 1/4$. 
\end{proof}

\begin{proposition}\label{prop:fareycoupons}
  Let $f(n)$ and $\set{N_t}$ be as in
  Lemma~\ref{lem:fareycoupons}. Then there exists a constant
  $\kappa'>0$ such that for almost every $P\in\Omega_f$, the following
  holds: for every non-trivial interval $I\subset[0,1]$, there exists
  $t_0(I,P)$ such that
  \begin{equation*}
    \#\parens*{\bigcup_{N_t < n \leq N_{t+1}}\QQ(P_n)\cap I} \geq \kappa' F_t\lambda(I)
  \end{equation*}
  for all $t\geq t_0(I, P)$.
\end{proposition}

\begin{proof}
  Let $I\subset[0,1]$ be a non-trivial interval. Notice that
  \begin{equation}\label{eq:notice}
    \#\parens*{\bigcup_{N_t < n \leq N_{t+1}}\QQ(P_n)\cap I} \geq X_t (I),
  \end{equation}
  so that it will suffice to show that
  $X_t(I) \geq \kappa' F_t \lambda(I)$. As in the proof of
  Lemma~\ref{lem:P(L)=1}, Chebyshev's inequality implies that
  \begin{equation*}
    \PP_f\parens*{X_t(I) < \frac{1}{2}\EE(X_t(I))} \leq 4\frac{\sigma^2(X_t(I))}{\EE(X_t(I))^2}.
  \end{equation*}
  Then, from Lemma~\ref{lem:fareycoupons}, we have
  \begin{equation*}
    \PP_f\parens*{X_t(I) < \frac{1}{2}\EE(X_t(I))} \ll \frac{1}{F_t}, 
  \end{equation*}
  and since $\sum_t F_t^{-1}$ converges as $t\to\infty$, we find by
  the Borel--Cantelli Lemma that the probability is zero that
  $X_t(I) < \frac{1}{2}\EE(X_t(I))$ for infinitely many $t$. In other
  words, there is a full-measure subset $U_I\subset\Omega_f$ such that
  for every $P\in U_I$, we have
  $X_t(I) \geq \frac{1}{2}\EE(X_t(I)) \geq C F_t\lambda(I)$ for all
  sufficiently large $t$, by Lemma~\ref{lem:fareycoupons}.

Let 
  \begin{equation*}
    U = \bigcap_{I \textrm{ with endpoints in } \QQ} U_I. 
  \end{equation*}
  Then $\PP_f(U)=1$. For every $P\in U$, there exists $\kappa'>0$ such
  that for any interval $I\subset[0,1]$, we have
  $X_t(I) \geq \kappa' F_t \lambda(I)$. The proposition is proved, in
  view of~(\ref{eq:notice}).
\end{proof}

\begin{proposition}[Random local ubiquity]\label{prop:randomlocalubiquity}
  Suppose $0\leq f(n)\leq n$ is a sequence whose average order is at
  least linear, and let $\set{N_t}$ be as in
  Lemma~\ref{lem:averageorder}. Then for $\PP_f$-almost every
  $\calP\in\Omega_f$ there exists a constant $\kappa>0$ such that for
  any interval $I\subset\RR$, there exists $t_0:=t_0(I)$ with
  \begin{equation*}
    \lambda\parens*{I\cap\bigcup_{N_t < n \leq N_{t+1}} \bigcup_{a\in\calP_n}B\parens*{\frac{a}{n}, \frac{1}{F_t}}} \geq \kappa \lambda(I) \quad\textrm{for all } t\geq t_0.
  \end{equation*}
  That is, $(\QQ(\calP), J(P))$ is almost surely a locally
  ubiquitous system relative to
  $\parens*{\rho, \set{N_t}}$, where $\rho(n) = F_t^{-1}$
  whenever $n\in (N_t, N_{t+1}]$.
\end{proposition}

\begin{proof}
  By Proposition~\ref{prop:fareycoupons}, for almost every
  $P\in\Omega_f$ there exists a constant $\kappa'>0$ such that for any
  interval $I\subset [0,1]$, there exists $t_0 := t_0(I,P)$ such that
  \begin{equation*}
    \#\parens*{I\cap\bigcup_{N_t < n \leq N_{t+1}} \QQ(P_n)} \geq \kappa' F_t \lambda(I) \quad\textrm{for all } t\geq t_0,
  \end{equation*}
  and from the assumption in this proposition, we may safely assume
  that we have $F_t \geq a N_{t+1}^2\, (a>0),$ for $t\geq t_0$, too.
  Recall that the distance between any two points in $\calF_{N_{t+1}}$
  is greater than $N_{t+1}^{-2}$. Therefore, if we center an interval
  of radius $F_t^{-1}$ around every point of 
  \begin{equation*}
   I\cap\bigcup_{N_t < n \leq N_{t+1}} \QQ(P_n),
  \end{equation*}
  the measure of the resulting union of open intervals will be bounded
  below by in the following way: 
  \begin{multline*}
    \lambda\parens*{I\cap\bigcup_{N_t < n \leq N_{t+1}} \QQ(P_n) + \parens*{-\frac{1}{F_t},\frac{1}{F_t}}} \\ \geq \lambda\parens*{I\cap\bigcup_{N_t < n \leq N_{t+1}} \QQ(P_n) + \parens*{-\frac{1}{2N_{t+1}^2},\frac{1}{2N_{t+1}^2}}}\\ \geq \kappa' \lambda(I) \frac{F_t}{N_{t+1}^2} \geq a\kappa' \lambda(I),
  \end{multline*}
  as long as $t\geq t_0(I,P)$. This proves the proposition.
\end{proof}

%=========================================

\section{Random local ubiquity when $f(n)$ is positive and bounded on average}
\label{sec:bounded}

The purpose of this section is to show that if the average order of
$f(n)$ is positive and bounded, then for $\PP_f$-almost every
$P\in\Omega_f$, the pair $(\QQ(P), J(P))$ are a locally ubiquitous
system.

The following lemma is analogous to Lemma~\ref{lem:averageorder}. 

\begin{lemma}\label{lem:boundedonaverage}
  Let $0\leq f(n) \leq n$ be an integer sequence. Its average order is
  bounded and positive if and only if there exists an integer sequence
  $\set{N_t = k^t}$, where $k>1$, such that $N_{t+1}\ll F_t \ll N_t$
  for all sufficiently large $t$. Furthermore, one can choose $k$ large
  enough to ensure that there exists some positive $\lambda < 1$ such
  that $F_t \leq \lambda F_{t+1}$ for all sufficiently large $t$.
\end{lemma}

\begin{proof}
  First suppose $0 < a< b$ and $aN \leq \sum_{n=1}^Nf(n) \leq bN$
  for all sufficiently large $N$. Let $k > b/a$ be an integer and let
  $N_t = k^t$. Note then that
  \begin{align*}
    F_t &= \sum_{n=1}^{N_{t+1}}f(n) - \sum_{n=1}^{N_t}f(n) \geq aN_{t+1} - bN_t = \parens*{a - \frac{b}{k}} N_{t+1},
  \end{align*}
  as long as $t$ is large. (In fact, if we choose $k$ large enough we
  can ensure that the parenthetical quantity exceeds, say, $a/2$.) On
  the other hand, we have $F_t \leq b N_{t+1}$ for all sufficiently
  large $t$, immediately. From here, we can prove the ``furthermore''
  claim analogously to how it was proved in
  Lemma~\ref{lem:averageorder}.

  Now, suppose that there is an integer sequence $\set{N_t=k^t}$,
  $k>1$, and constants $0<a' <b'$ such that
  $a' N_{t+1} \leq F_t \leq b'N_t$ for all sufficiently large $t$. Let
  $N_t \leq N < N_{t+1}$. On one hand we have
  \begin{equation*}
    \sum_{n=1}^Nf(n) \geq F_{t-1} \gg N_t \gg N
  \end{equation*}
for large $t$. On the other we have 
\begin{equation*}
  \sum_{n=1}^N f(n) \leq F_0 + F_1 + \dots + F_t \leq b' (1 + k + \dots + k^t) \ll k^{t+1} \ll N
\end{equation*}
for large $t$.
\end{proof}

Next, we establish a basic arithmetic fact. 

\begin{lemma}\label{lem:elementary}
  For any $n\in\NN$ and integer $m\geq 0$ we have
  \begin{equation*}
    1 - \frac{m}{n} \leq \parens*{1 - \frac{1}{n}}^m.
  \end{equation*}
\end{lemma}

\begin{proof}
  First, note that the claim is true whenever $m=0$ or $m=1$. It is
  also true whenever $n=1$.

  Now, let $n\geq 2$. Consider the function
  \begin{equation*}
    g(x) = \parens*{1-\frac{1}{n}}^x - \parens*{1 - \frac{x}{n}}.
  \end{equation*}
  We will show that $g(x)\geq 0$ for all $x\geq 1$. Since we already
  know that $g(1)\geq 0$, it is enough to show that $g'(x)\geq 0$ for
  all $x\geq 1$. Note that
  \begin{equation*}
    g'(x) = \parens*{1-\frac{1}{n}}^x\log\parens*{1-\frac{1}{n}}  + \frac{1}{n},
  \end{equation*}
and that 
  \begin{equation*}
    g''(x) = \parens*{1-\frac{1}{n}}^x\brackets*{\log\parens*{1-\frac{1}{n}}}^2.
  \end{equation*}
  Clearly, for any $x\geq 1$, we have $g''(x)\geq 0$. So
  we only have to show that $g'(1) \geq 0$.

  To see this, consider the function $h(y) = y\log y + (1-y)$ and note
  that $h(1) = 0$. Notice also that $h'(y) = \log y$ and is therefore
  nonpositive whenever $y\in(0,1]$. Therefore, $h(y)\geq 0$ for all
  $y\in (0,1]$. In particular, $g'(1)\geq 0$. 
\end{proof}

The following proposition is the version of
Proposition~\ref{prop:randomlocalubiquity} for the case where $f(n)$
is positive on average and bounded on average.

\begin{proposition}[Random local ubiquity]\label{prop:expectedubiquitybddsubseq}
  Suppose $0\leq f(n)\leq n$ is a sequence whose average order is
  positive and bounded, and let $\set{N_t}$ be as in
  Lemma~\ref{lem:boundedonaverage}. Then there exist $\kappa>0$ such
  that for almost every $P\in \Omega_f$, the following holds: for
  every non-trivial interval $I\subset [0,1)$ we have
    \begin{equation*}
      \lambda\parens*{I\cap \bigcup_{N_t < n \leq N_{t+1}} \bigcup_{a\in P_n}B\parens*{\frac{a}{n}, \frac{1}{F_t}}} \geq \kappa \lambda(I)
  \end{equation*}
  whenever $t\geq t_0(I)$.  That is, $(\QQ(\calP), J(P))$ is
  almost surely a locally ubiquitous system relative to
  $\parens*{\rho, \set{N_t}}$, where $\rho(n) = F_t^{-1}$
  whenever $n\in (N_t, N_{t+1}]$.
\end{proposition}

\begin{proof}
  Let $I\subset [0,1]$ be a nontrivial subinterval. Let $t_0:=t_0(I)$
  be such that for any $t\geq t_0$, there exist $\ell_1<\ell_2 \in [N_t]$ such that we have
  \begin{equation*}
    I \supset I' := \bigcup_{\ell=\ell_1+1}^{\ell_2} \left[\frac{\ell}{N_t}, \frac{\ell+1}{N_t}\right) = \bigcup_{\ell=\ell_1+1}^{\ell_2}I_\ell
  \end{equation*}
  where $\lambda(I') \geq \frac{\lambda(I)}{2}$. Notice that
  $\lambda(I') = (\ell_2-\ell_1)/N_t$. Notice also that $t_0$ really only
  depends on the length of $I$, and not on $I$ itself.

  For each $N_t < n \leq N_{t+1}$, we make an independent uniform random
  choice $\calP_n\in \Omega_{f,n}$. Define the random variable
  $X_{n,\ell} = \#(\QQ(P_n) \cap I_\ell)$. For any $\ell\in[N_t]$, and
  any integer $n \in (N_t, N_{t+1}]$, we have
  \begin{equation*}
     \#\parens*{\frac{[n]}{n} \cap I_\ell} \geq 1
  \end{equation*}
  so
  \begin{equation*}
    \PP_f(X_{n,\ell} = 0) \leq 1-\frac{f(n)}{n} \overset{\textrm{Lem.~\ref{lem:elementary}}}{\leq} \parens*{1 - \frac{1}{n}}^{f(n)} \leq \parens*{1 - \frac{1}{N_{t+1}}}^{f(n)}.
  \end{equation*}
  Since choices are made independently,
  \begin{equation*}
    \PP_f(X_{n,\ell}=0, \, N_t < n\leq N_{t+1}) \leq \parens*{1-\frac{1}{N_{t+1}}}^{F_t}.
  \end{equation*}
  Let
  \begin{equation*}
    Y_\ell = \min\set*{1, \sum_{n\in(N_t, N_{t+1}]} X_{n,\ell}},
  \end{equation*}
  and let $Z_t = \sum_{\ell=\ell_1+1}^{\ell_2} Y_\ell$. That is,
  $Y_\ell$ is $1$ or $0$ according as $I_\ell$ witnesses the choice of
  some $a\in P_n$, with $n\in (N_t, N_{t+1}]$. And $Z_t$ is the number
  of intervals $I_{\ell_1+1}, \dots, I_{\ell_2}$ containing points
  $a\in P_n$, with $n\in (N_t, N_{t+1}]$.

  Since $Y_\ell$ only takes values $0$ or $1$, we have
  \begin{equation*}
    \EE(Y_\ell) = \PP_f(Y_\ell = 1) \geq 1 - \parens*{1-\frac{1}{N_{t+1}}}^{F_t}.
  \end{equation*}
  Therefore, 
  \begin{equation*}
    \EE(Z_t) = \sum_{\ell=\ell_1+1}^{\ell_2} \EE(Y_\ell) \geq (\ell_2-\ell_1)\brackets*{1-\parens*{1-\frac{1}{N_{t+1}}}^{F_t}}.
  \end{equation*}
  Since $N_{t+1} \asymp F_t$, the bracketed factor is universally
  bounded below. That is, we have $\EE(Z_t) \geq C_{f} (\ell_2-\ell_1)$, where
  $C_{f}>0$ is a constant only depending on $f$ (or, more
  appropriately, the sequence $\set{N_t}$).

  As for the variance,
  \begin{align*}
    \sigma^2(Z_t) &= \EE(Z_t^2) - \EE(Z_t)^2 \\
                  &= \sum_{k, \ell=\ell_1+1}^{\ell_2} \EE(Y_kY_\ell) - \parens*{\sum_{\ell=\ell_1+1}^{\ell_2} \EE(Y_\ell)}^2\\
                  &= \sum_{k\neq \ell }\brackets*{\EE(Y_kY_\ell) - \EE(Y_k)\EE(Y_\ell)} + \sum_{\ell = a+1}^b \brackets*{\EE(Y_\ell^2)  - \EE(Y_\ell)^2}.
  \end{align*}
  Notice that $Y_k, Y_\ell, Y_kY_\ell$ are random variables taking only the
  values $0$ and $1$, hence, their expected values are exactly the
  probabilities that they take the value $1$. In particular,
  \begin{align*}
    \EE(Y_kY_\ell) - \EE(Y_k)\EE(Y_\ell) &= \PP_f(Y_k=1, Y_\ell=1) - \PP_f(Y_k=1)\PP_f(Y_\ell=1) \\
                                         &= \brackets*{\PP_f(Y_k=1\mid Y_\ell=1) - \PP_f(Y_k=1)}\PP_f(Y_\ell=1).
  \end{align*}
  Note that the bracketed factor is non-positive: knowing that the
  interval $I_\ell$ is chosen at some point of our process does not
  increase the probability that $I_k$ is chosen. Therefore, we may
  conclude that 
  \begin{align*}
    \sigma^2 (Z_t) &\leq \sum_{\ell = \ell_1+1}^{\ell_2} \brackets*{\EE(Y_\ell^2)  - \EE(Y_\ell)^2} \\
                   &\leq \sum_{\ell = \ell_1+1}^{\ell_2} \EE(Y_\ell)\\
                   &= \EE(Z_t),
  \end{align*}
  for all $t\geq t_0(I)$.

  We now argue as in the proof of
  Proposition~\ref{prop:fareycoupons}. Chebyshev's inequality implies that
  that
  \begin{equation*}
    \PP_f\parens*{Z_t < \frac{1}{2}\EE(Z_t)} \leq 4\frac{\sigma^2(Z_t)}{\EE(Z_t)^2} \leq \frac{4}{\EE(Z_t)} \ll \frac{1}{\ell_2 - \ell_1} \ll \frac{1}{N_t}.
  \end{equation*}
  And since $\sum_t N_t^{-1}$ converges as $t\to\infty$, the
  Borel--Cantelli Lemma implies that the probability is zero that
  $Z_t < \frac{1}{2}\EE(Z_t)$ for infinitely many $t$. In other words,
  there is a full-measure subset $U_I\subset\Omega_f$ such that for
  every $P\in U_I$, we have
  \begin{equation*}
    Z_t \geq \frac{1}{2}\EE(Z_t) \geq \frac{C_f}{2}(\ell_2 - \ell_1) = \frac{C_f}{2}N_t\lambda(I')\geq\frac{C_f}{4}N_t\lambda(I)
  \end{equation*}
  for all sufficiently large $t$. Intersecting over all intervals
  $I\subset [0,1]$ with rational endpoints, we conclude that there
  exists some constant $\kappa_f>0$ and full-measure subset
  $U\subset \Omega_f$ such that for every $P\in U$, the following
  holds: for every interval $I\subset[0,1]$, there is some
  $t_0:=t_0(I,P)$ such that
  $Z_t(I) \geq \kappa_f N_t\lambda(I)$ for all
  $t\geq t_0$.

  Finally, let us consider the measure of
  \begin{equation*}
    I \cap \bigcup_{N_t<n\leq N_{t+1}}\bigcup_{a\in\calP_n}B\parens*{\frac{a}{n}, \frac{1}{F_t}}.
  \end{equation*}
  We have $F_t \ll N_{t+1} \ll N_t$, and therefore the measure of the
  above set is bounded below by
  \begin{equation*}
    \kappa_f''\lambda\parens*{I \cap \bigcup_{N_t<n\leq N_{t+1}}\bigcup_{a\in\calP_n}B\parens*{\frac{a}{n}, \frac{1}{N_t}}}
  \end{equation*}
  where $\kappa_f''>0$ is a constant only depending on $f$. The
  observations in the previous paragraph in turn give us the bound
  \begin{equation*}
    \lambda\parens*{I \cap \bigcup_{N_t<n\leq N_{t+1}}\bigcup_{a\in\calP_n}B\parens*{\frac{a}{n}, \frac{1}{N_t}}} \geq \kappa_f' \lambda(I).
  \end{equation*}
  whenever $t\geq t_0(I,P)$. This proves the proposition, with
  $\kappa_f = \kappa_f'\kappa_f''$.
\end{proof}

%=================================================
%=================================================
%=================================================
%=================================================

\section{Proof of Theorem~\ref{thm:boundedandlinear}}
\label{sec:general}

In \S\ref{sec:phiubiq} and \S\ref{sec:bounded} we showed that one
almost surely obtains a locally ubiquitous system by choosing rational
numbers according to the hypotheses of
Theorem~\ref{thm:boundedandlinear}. We are now prepared to state a
proof of that theorem.

\begin{proof}[Proof of Theorem~\ref{thm:boundedandlinear}]
  Let $\set{N_t}$ be as in Lemma~\ref{lem:averageorder}
  or~\ref{lem:boundedonaverage}, whichever of the two is relevant. The
  function $\rho(n)$ defined by
  \begin{equation*}
    \rho(n) = \frac{1}{F_t} \quad\textrm{whenever}\quad n\in(N_t, N_{t+1}]
  \end{equation*}
  is $\set{N_t}$-regular.

  Proposition~\ref{prop:randomlocalubiquity} (or
  Proposition~\ref{prop:expectedubiquitybddsubseq}) ensures that there
  is a full-measure set $U\subset \Omega_f$ such that for every
  $\calP\in U$, the pair $(\QQ(P), J(P))$ is a locally ubiquitous
  system relative to $\rho(n)$ and $\set{N_t}$. We are in a position
  to apply Theorem~\ref{thm:ubiquity}.

  Let $\Psi:\NN\to\RR^+$ be decreasing. Then
  \begin{align*}
    \sum_{n\in(N_t, N_{t+1}]} f(n)\Psi(n) &\leq \Psi(N_t)\sum_{n\in(N_t, N_{t+1}]}f(n)  \\
                                          &= F_t \Psi(N_t) \ll F_{t-1}\Psi(N_t) \\
                                          &= \frac{\Psi(N_t)}{\rho(N_t)}
  \end{align*}
  for large $t$. (The comparison $F_t \ll F_{t-1}$ comes directly from
  Lemma~\ref{lem:averageorder} or~\ref{lem:boundedonaverage}.)
  Therefore, if $\sum_n f(n)\Psi(n)$ diverges, then the sum
  $\sum_t \frac{\Psi(N_t)}{\rho(N_t)}$ also diverges. By
  Theorem~\ref{thm:ubiquity}, we have $\lambda(W^\calP(\Psi))=1$ as
  desired.

  On the other hand, if $\Psi:\NN\to\RR_{\geq 0}$ is such that
  $\sum_n f(n)\Psi(n)$ converges, then it is a consequence of the
  Borel--Cantelli Lemma that $\lambda(W^\calP(\Psi))=0$ for every
  $P\in\Omega_f$, regardless of monotonicity of $\Psi$.
\end{proof}

  % ===========================================
 
\section{Proof of Theorem~\ref{thm:counterex}}
\label{sec:monotonicity}

Clearly, if $f(n) \gg n$, the Duffin--Schaeffer Counterexample will
serve its purpose for Theorem~\ref{thm:boundedandlinear}. In this
section we investigate how much slower $f(n)$ can grow. In a sense,
finding counterexamples for all $f(n) \gg n/\log\log n$ is already out
of the question, because $\varphi(n) \gg n/\log\log n$ and we do not
expect any counterexamples in this case. We also should not expect to
find a counterexample every time there is a sequence of $n$'s on which
$f(n) \ll n/\log\log n$ fails, because this is true for $\varphi$, on
the sequence of primes. We can, however, prove:
\begin{itemize}
\item Theorem~\ref{thm:counterex}, which provides counterexamples
  whenever $f(n) = n/o(\log\log n)$, and
\item Corollary~\ref{cor:sweetspot}, which shows that there exist
  functions such that $f(n) \ll n/\log\log n$ on a sequence of $n$'s,
  for which there are counterexamples.
\end{itemize}
These show that $\varphi$ is in a sort of a ``sweet spot'' for this
problem.

Before stating the proof of Theorem~\ref{thm:counterex}, we define
\begin{equation*}
  \calA_n(\Psi) := \set*{\real\in [0,1] : \Abs{\real - \frac{a}{n}}<\Psi(n) \textrm{ for some } a\in [n]},
\end{equation*}
so that $W(\Psi) = \limsup_{n\to\infty} \calA_n$.

\begin{proof}[Proof of Theorem~\ref{thm:counterex}]
  Since
  \begin{equation*}
    \frac{f(n)\log\log n}{n} \longrightarrow \infty,
  \end{equation*}
  we can find some non-increasing sequence $\tau(n)\to 0$ such that
  \begin{equation}\label{eq:fbound}
    f(n) \geq \frac{n}{\tau(n)\log\log n}
  \end{equation}
  for all $n$ sufficiently large. Now, let $\sum_j c_j$ be a
  convergent series of positive real numbers. Since
  $\tau(n)\downarrow 0$, we can find a strictly increasing sequence
  $\set{M_j}_{j\geq 0}$ of integers with $M_0=0$, such that
  \begin{equation}
    \label{eq:Mj}
    \frac{c_j}{\tau((M_j-1)!)} \gg 1  \quad\textrm{holds for infinitely many}\quad j\geq 1.
  \end{equation}
  Let us further stipulate that
  \begin{equation}
    \textrm{any two consecutive terms in the
      sequence $\set{M_j}$ differ by at least $2$.}\label{eq:diff2}
  \end{equation}
  Now set $K_j = M_j!$ for $j\geq 1$, and let
  \begin{equation*}
    k_i^{(j)} =\frac{K_j}{i}
  \end{equation*}
  for $i=1, \dots, M_j$.  Notice that
  \begin{equation}
    \label{eq:k<k}
    M_j! = k_1^{(j)} > k_2^{(j)} > \dots > k_{M_j}^{(j)} =(M_j-1)!
  \end{equation}
  so that the $k_i^{(j)}$ are pairwise distinct. (The case
  $(M_j-1)! = M_{j-1}!$ is ruled out by stipulation~(\ref{eq:diff2}).)
  Define
  \begin{equation}\label{eq:defpsi}
    \Psi\parens{k_i^{(j)}} = \frac{c_j}{K_j},
  \end{equation}
  and $\Psi(k)=0$ for all $k\notin\set{k_i^{(j)}}$. Then we have
  $\calA_{k_i^{(j)}}(\Psi) \subset \calA_{k_1^{(j)}}(\Psi)$ for all
  $j\geq 1$ and $i=1, \dots, M_j$. In particular,
  \begin{equation*}
    \calW(\Psi) =  \limsup_{n\to\infty}\calA_n(\Psi) \subset \limsup_{j\to\infty} \calA_{k_1^{(j)}}(\Psi).
  \end{equation*}
  But
  $\sum_j \lambda\parens{\calA_{k_1^{(j)}}(\Psi)} \leq \sum_j 2c_j <
  \infty$,
  so the Borel--Cantelli Lemma implies that $\lambda\parens*{\calW(\Psi)} = 0$.

  It is only left to show that
  \begin{equation*}
    \sum_{n\geq 1}f(n)\Psi(n) = \infty.
  \end{equation*}
  For some large $j_0$ we have
  \begin{align}
    \sum_{n\geq 1}f(n)\Psi(n) &\overset{(\ref{eq:defpsi})}{=} \sum_{j\geq 1}\sum_{i=1}^{M_j}\frac{f(k_i^{(j)})c_j}{K_j} \nonumber \\
                                        &\overset{(\ref{eq:fbound})}{\geq} \sum_{j\geq j_0}\sum_{i=1}^{M_j}\frac{c_j\,k_i^{(j)}}{\tau(k_i^{(j)}) K_j \log\log k_i^{(j)}} \nonumber \\
                                        &\overset{(\ref{eq:k<k})}{\geq} \sum_{j\geq j_0}\frac{c_j}{\tau(k_{M_j}^{(j)})\log\log k_1^{(j)}}\sum_{i=1}^{M_j}\frac{1}{i} \nonumber \\
                                        &\geq \sum_{j\geq j_0}\frac{c_j\log M_j}{\tau((M_j-1)!) \log\log (M_j!)}.\label{eq:mustdiv}
  \end{align}
  Notice that
  \begin{equation*}
    \frac{c_j\log M_j}{\tau((M_j-1)!) \log\log (M_j!)} \sim  \frac{c_j}{\tau((M_j-1)!)} \quad (j\to\infty).
  \end{equation*}
  Therefore, by~(\ref{eq:Mj}), the sum~(\ref{eq:mustdiv}) must
  diverge, and this proves the theorem.
\end{proof}

\begin{remark}
  To see that this could not possibly work for $f=\varphi$, we compute
  \begin{equation*}
    \sum_{n\geq 1}\varphi(n)\Psi(n) = \sum_{j\geq
      1}\sum_{i=1}^{M_j}\frac{\varphi(k_i^{(j)})c_j}{K_j}
    \leq \sum_{j\geq 1}\frac{c_j}{K_j}\sum_{k\mid K_j}\varphi(k) = \sum_{j\geq 1}c_j,  
  \end{equation*}
  which converges.
\end{remark}

\begin{proof}[Proof of Corollary~\ref{cor:sweetspot}]
  We re-create the previous proof, starting with a convergent series
  $\sum_j c_j$ of positive real numbers. Let $\tau(n)$ be a decreasing
  function that converges to $0$, and such that
  \begin{equation*}
    \tau(n) \geq \frac{1}{\log\log n}    
  \end{equation*}
  for all sufficiently large $n$. We find a strictly increasing
  sequence $\set{M_j}_{j\geq 0}$ of integers with $M_0=0$, such that
  \begin{equation}
    \label{eq:Mjtwo}
    \frac{c_j}{\tau((M_j-1)!)} \gg 1  \quad\textrm{holds for infinitely many}\quad j\geq 1.
  \end{equation}
  We further stipulate that
  \begin{equation}
    \textrm{any two consecutive terms in the
      sequence $\set{M_j}$ differ by at least $2$.}\label{eq:diff2two}
  \end{equation}
  We define $K_j, k_i^{(j)}, \Psi$ as before. Again, for the same
  reasons as before, we have $\abs{\calW(\psi)} = 0$.

  Now suppose $f$ is such that
  \begin{equation*}
    \begin{dcases}
      f(n) \geq \frac{n}{\tau(n)\log\log n} &\textrm{if $n=k_i^{(j)}$ for some $i,j$, and}\\
      f(n) \ll \frac{n}{\log\log n} &\textrm{otherwise}.
    \end{dcases}
  \end{equation*}
  The previous proof will now show that
  $\sum_{n\geq 1}f(n)\Psi(n)$ diverges.
\end{proof}

\section{Proof that Conjecture~\ref{conj:randomdsc} implies
  Catlin's Conjecture}
\label{sec:wdsc=catlin}

The purpose of this section is to show 
\begin{equation}\label{eq:ifrandomthencatlin}
  \textrm{Conjecture~\ref{conj:randomdsc}} \quad\implies\quad \textrm{Catlin's Conjecture.}
\end{equation}
Consider the following weak version of the Duffin--Schaeffer
Conjecture.
\begin{conjecture}[Weak Duffin--Schaeffer Conjecture]\label{conj:weakdsc}
  For any $\Psi : \NN\to\RR_{\geq 0}$ such that
  $\sum_{n=1}^\infty \varphi(n)\Psi(n) =\infty$, we have
  $\lambda\parens*{\calW (\Psi)} = 1$.
\end{conjecture}

\begin{remark}
  It is easy to see
  \begin{equation}\label{eq:randomimpliesweak}
    \textrm{Conjecture~\ref{conj:randomdsc}} \quad\implies\quad \textrm{Conjecture~\ref{conj:weakdsc}.}
  \end{equation}
  In fact, even if a single point in $\Omega_\varphi$
  satisfies~(\ref{eq:rkt}) without the monotonicity assumption, then
  Conjecture~\ref{conj:weakdsc} follows. In particular, the
  Duffin--Schaeffer Conjecture implies
  Conjecture~\ref{conj:weakdsc}. But even weaker statements would also
  imply Conjecture~\ref{conj:weakdsc}. For instance, the following
  would be enough: For each $\Psi:\NN\to\RR_{\geq 0}$ such that
  $\sum\varphi(n)\Psi(n)=\infty$, there exists a point
  $P\in \Omega_\varphi$ (which may depend on $\Psi$) with
  $\lambda(W^P(\Psi))=1$.
\end{remark}

We will show
\begin{equation}\label{eq:weakiffcatlin}
  \textrm{Conjecture~\ref{conj:weakdsc}} \quad\iff\quad \textrm{Catlin's Conjecture}
\end{equation}
which, together with~(\ref{eq:randomimpliesweak}),
establishes~(\ref{eq:ifrandomthencatlin}).\footnote{The
  equivalence~(\ref{eq:weakiffcatlin}) answers a footnote
  of~\cite[Page 634]{ds_counterex}, where it was suggested that
  Conjecture~\ref{conj:weakdsc} and its relationship to the
  Duffin--Schaeffer Conjecture were unexplored questions. In fact,
  those questions \emph{have} been explored, in~\cite{CatlinConj},
  where Catlin introduced his conjecture! But more exploration is
  required.}

\begin{lemma}\label{lem:overlinePsi}
  Given a function $\Psi : \NN\to\RR_{\geq 0}$, define the new function
  \begin{equation*}
    \overline\Psi (n) := \max_{k\in\NN}\set*{\Psi(kn)}. 
  \end{equation*}
  Then $W(\Psi) = W(\overline\Psi)$.
\end{lemma}

\begin{proof}
  First, we can safely assume that for each $n$, we have
  $\lim_{k\to\infty}\Psi(kn) = 0$, for otherwise we would have
  $W(\Psi) = W(\overline\Psi) = [0,1]$.  

  Since $\Psi \leq \overline\Psi$, we have
  $W(\Psi)\subseteq W(\overline\Psi)$, and we only have to show the
  other containment. Suppose that $x \in W(\overline\Psi)$, meaning
  that there are infinitely many $(a,n)\in \ZZ\times\NN$ satisfying
  \begin{equation}\label{eq:theabove}
    \abs*{x - \frac{a}{n}} < \overline\Psi(n). 
  \end{equation}
  For each $n$, let $k:=k_n$ denote the integer maximizing
  $\Psi(kn)$. Such an integer is guaranteed to exist, by the
  assumption made in the previous paragraph. Now, for each $(a,n)$
  satisfying~(\ref{eq:theabove}), we have 
  \begin{equation}\label{eq:morerecent}
    \abs*{x - \frac{k_na}{k_nn}} < \Psi(k_n n). 
  \end{equation}
  Notice that $\set{k_n n}_{n\in\mathcal N}$ is an infinite set for
  any subsequence $\mathcal N\subseteq \NN$, because $k_n n$ can be
  made arbitrarily large by taking $n$ large. In particular, the
  infinitude of solutions $(a,n)$ to~(\ref{eq:theabove}) gives rise to
  an infinitude of solutions $(k_na, k_n n)$
  to~(\ref{eq:morerecent}). Therefore,
  $W(\overline\Psi)\subseteq W(\Psi)$ and we are done.
\end{proof}

\begin{proof}[Proof of~(\ref{eq:weakiffcatlin})]
  Now, suppose Catlin's Conjecture is true, and suppose we are given a
  function $\Psi : \NN\to\RR_{\geq 0}$ such that
  $\sum \varphi(n)\Psi(n)$ diverges. Clearly, we also have that
  $\sum\varphi(n)\overline\Psi(n)$ diverges.  By the divergence part
  of Catlin's Conjecture, we get $\lambda(W(\Psi))=1$, which
  establishes Conjecture~\ref{conj:weakdsc}.

  On the other hand, suppose Conjecture~\ref{conj:weakdsc} is true,
  and let $\Psi: \NN\to\RR_{\geq 0}$ be such that
  $\sum\varphi(n)\overline\Psi(n)$ diverges. Then, by assumption, we
  have $\lambda(W(\overline\Psi))=1$, which by
  Lemma~\ref{lem:overlinePsi} implies that $\lambda(W(\Psi))=1$.
\end{proof}

Now~(\ref{eq:ifrandomthencatlin}) follows
from~(\ref{eq:randomimpliesweak}) and~(\ref{eq:weakiffcatlin}).\qed

\subsection*{Acknowledgments}
\label{acknowledgments}

I thank Christoph Aistleitner for bringing Catlin's Conjecture to my
attention, and Ayan Chakraborty for his comments on the
previous version of this paper.

This was the first project that I started after my daughter, Luna, was
born. I would like to acknowledge my wife, Yeni, for her partnership.

% ===========================================

% ========================================

\bibliographystyle{plain}

%\bibliography{../bibliography}

\end{document}